\documentclass{amsart}
\usepackage[british]{babel}
\usepackage{graphicx}
\usepackage{amssymb}
\usepackage{mathrsfs}
\usepackage{paralist}
\usepackage[all]{xy}

\newtheorem{theorem}{Theorem}[section]
\newtheorem{lemma}[theorem]{Lemma}
\newtheorem{proposition}[theorem]{Proposition}
\newtheorem{corollary}[theorem]{Corollary}
\newtheorem{conjecture}[theorem]{Conjecture}
\newtheorem{letterthm}{Theorem}
\newtheorem{lettercor}[letterthm]{Corollary}
\newtheorem{letterprop}[letterthm]{Proposition}

\theoremstyle{definition}
\newtheorem{definition}[theorem]{Definition}
\newtheorem{example}[theorem]{Example}
\newtheorem{remark}[theorem]{Remark}

\newcommand{\sect}{\ensuremath{\S}~}

\newcommand{\abs}[1]{\ensuremath{\lvert#1\rvert}}
\newcommand{\angles}[1]{\ensuremath{\langle#1\rangle}}

\newcommand{\field}[1]{\ensuremath{\mathbf{#1}}}
\newcommand{\Z}{\field{Z}}
\newcommand{\R}{\field{R}}
\newcommand{\Q}{\field{Q}}

\newcommand{\E}{\field{E}}

\newcommand{\st}{\; {\big\vert} \;}
\newcommand{\Pres}[2]{\big\langle#1\st#2\big\rangle}
\newcommand{\ab}[1]{\ensuremath{#1_{\text{ab}}}}

\newcommand{\lift}[1]{\ensuremath{\widetilde{#1}}}

\newcommand{\ZG}{\ensuremath{\Z\Gamma}}
\newcommand{\KXone}[1]{\ensuremath{\mathrm{K}(#1,1)}}

\newcommand{\FPnn}[1]{\ensuremath{\mathrm{FP}_{#1}}}

\newcommand{\FP}{\ensuremath{\mathrm{FP}}}

\DeclareMathOperator{\Map}{Map}
\DeclareMathOperator{\rk}{rk}

\DeclareMathOperator{\defic}{def}
\DeclareMathOperator{\adef}{adef}

\DeclareMathOperator{\CAT}{CAT}

\begin{document}
% <<< metadata
\title[Presentations of subgroups of right-angled Artin groups]{Constructing
presentations of subgroups of right-angled Artin groups}

\author{Martin R.~Bridson}
\address{Dept.\ of Mathematics, Imperial College, London SW7 2AZ.}
\email{m.bridson@imperial.ac.uk}
\thanks{}

\author{Michael Tweedale}
\address{Dept.\ of Mathematics, University of Bristol, Bristol BS8 1TW.}
\email{m.tweedale@bristol.ac.uk}
\thanks{This work was supported in part by grants from the EPSRC. The first
author is also supported  by a Royal Society Wolfson Research Merit Award.}

\date{5th September 2007}
% >>>
\begin{abstract}
  Let $G$ be the right-angled Artin group associated to the flag complex
  $\Sigma$ and let $\pi:G\to\Z$ be its canonical height function.  We
  investigate the presentation theory of the groups $\Gamma_n=\pi^{-1}(n\Z)$
  and construct an algorithm that, given $n$ and $\Sigma$, outputs a
  presentation of optimal deficiency on a minimal generating set, provided
  $\Sigma$ is triangle-free; the deficiency tends to infinity as $n\to\infty$
  if and only if the corresponding Bestvina--Brady kernel $\bigcap_n\Gamma_n$
  is not finitely presented, and the algorithm detects whether this is the
  case. We explain why there cannot exist an algorithm that constructs finite
  presentations with these properties in the absence of the triangle-free
  hypothesis.  We explore what \emph{is} possible in the general case,
  describing how to use the configuration of $2$-simplices in $\Sigma$ to
  simplify presentations and giving conditions on $\Sigma$ that ensure that
  the deficiency goes to infinity with $n$. We also prove, for general
  $\Sigma$, that the abelianized deficiency of $\Gamma_n$ tends to infinity if
  and only if $\Sigma$ is $1$-acyclic, and discuss connections with the
  relation gap problem.
\end{abstract}
\maketitle\let\languagename\relax\sloppy
%         ^ ^ ^ ^ ^ ^ ^ ^ ^ ^ ^ ^---- kludge: cf.
% http://www.latex-project.org/cgi-bin/ltxbugs2html?pr=amslatex/3756

\section{Introduction}\label{sec:intro}
\noindent Right-angled Artin groups, or graph groups as they used to be known,
have been the object of considerable study in recent years, and a good picture
of their properties has been built up through the work of many different
researchers. For example, from S.~Humphries~\cite{Humphries} one knows that
right-angled Artin groups are linear; their integral cohomology rings were
computed early on by K.~Kim and F.~Roush~\cite{KimRoush}, and 
C.~Jensen and J.~Meier~\cite{JensenMeier} have extended this to include
cohomology with group ring coefficients. More recently, S.~Papadima and
A.~Suciu~\cite{PapSuciu} have computed the lower central series, Chen groups
and resonance varieties of these groups, while R.~Charney, J.~Crisp and
K.~Vogtmann~\cite{CCV} have explored their automorphism groups (in the
triangle-free case) and M.~Bestvina, B.~Kleiner and M.~Sageev~\cite{BKS} their
rigidity properties.

However, the feature of these groups that has undoubtedly been the most
significant in fuelling interest in them is their rich geometry.
In~\cite{CharneyDavis}, R.~Charney and M.~Davis construct for each
right-angled Artin group an Eilenberg--Mac~Lane space which is a compact,
non-positively curved, piecewise-Euclidean cube complex. This invitation to
apply geometric methods to the study of right-angled Artin groups was taken up
with remarkable effect by M.~Bestvina and N.~Brady~\cite{BestvinaBrady}.

One can parametrize right-angled Artin groups by finite simplicial complexes
$\Sigma$ satisfying a certain \emph{flag} condition. The Artin group
associated to $\Sigma$ depends heavily on the combinatorial structure of
$\Sigma$, not just its topology. However, each right-angled Artin group has a
canonical map onto $\Z$, and if one passes to the kernel of this map then
Bestvina and Brady show that the cohomological finiteness properties of such a
kernel are determined by the topology of $\Sigma$ alone.
(See~\sect\ref{sec:backgroundBB} for a precise statement.) 

In low dimensions, the cohomological properties of a group are intimately
connected to its presentation theory, so one might hope to see directly how
presentations of these Bestvina--Brady kernels are related to the
corresponding flag complex $\Sigma$. This point of view was adopted by
W.~Dicks and I.~Leary in~\cite{DicksLeary}; we embrace and extend it here.

To prove their theorem, Bestvina and Brady use global geometric methods.  Our
aim is to understand the behaviour of subgroups of right-angled Artin groups
at a more primitive, algorithmic level.  Our main focus will be the
\emph{algorithmic} construction of finite presentations for certain
approximations to the Bestvina--Brady kernels. It turns out that there are
profound reasons why such an approach can only take one so far, and so
philosophically one can conclude that some extra input (for example, from
geometry) is essential for a complete understanding of these groups: see
\sect\ref{sec:limitations}.

Let us now describe our results. Fix a connected finite flag complex $\Sigma$.
The principal objects of study in this paper are finite-index subgroups of the
corresponding right-angled Artin group $G=G_{\Sigma}$ that interpolate between
the well-understood group $G$ and the often badly-behaved Bestvina--Brady
kernel $H=H_{\Sigma}$. Specifically, if $\pi:G\to\Z$ is the canonical
surjection (see~\sect\ref{sec:backgroundBB}), so that $H=\ker\pi$, then we
consider the groups $\Gamma_n=\pi^{-1}(n\Z)$: thus $G=H\rtimes\Z$ and
$\Gamma_n=H\rtimes n\Z$.

Our expectation that these groups should have interesting presentation
theories comes from the Bestvina--Brady theorem.  Recall that the finiteness
property \FPnn{2} has sometimes been called \emph{almost finite
presentability} in the literature, because a group $\Gamma$ enjoys this
property if and only if it has a presentation $F/R$ with $F$ a finitely
generated free group and the abelian group $R/[R,R]$ finitely generated as a
module over the group ring $\ZG$, where the $\Gamma$-action is induced by the
conjugation action of $F$ on $R$. This $\ZG$-module is called the
\emph{relation module} of the presentation. For a long time it was an open
question whether or not almost finite presentability is in fact equivalent to
finite presentability, but one part of Bestvina and Brady's result implies
that this question has a negative answer: specifically, when $\Sigma$ is a
flag complex with non-trivial perfect fundamental group, they show that the
kernel $H_{\Sigma}$ is almost finitely presented but not finitely presented.

Motivated by this, we investigate the extent to which the topology of $\Sigma$
determines whether or not the number of relations needed to present $\Gamma_n$
remains bounded as $n$ increases, and similarly whether the number of
generators needed for the relation modules remains bounded. A natural
conjecture is that the number of relations required remains bounded if and
only if $H_{\Sigma}$ is finitely presented, while the rank of the relation
module remains bounded if and only if $H_{\Sigma}$ is almost finitely
presented. We prove the second part of this conjecture in this paper.  In the
light of this, a proof of the first part (which eludes us)  would establish
the existence of a finitely presented group with a \emph{relation gap},
without giving an explicit example.  (See~\cite{BTtorsion} for a fuller
discussion of the relation gap problem.)

Here is a summary of our results.

\begin{letterprop}[Proposition~\ref{th:vertpres}]\label{lth:vertpres}
  If $\Sigma$ is connected then for each integer $n\ge1$ there is a
  generating set for $\Gamma_n$ indexed by the vertices of $\Sigma$, and
  $\Gamma_n$ cannot be generated by fewer elements.
\end{letterprop}

The next two theorems are most cleanly phrased in the language of efficiency
and deficiency: see~\sect\ref{sec:backgrounddef}.

\begin{letterthm}[Theorems~\ref{th:graphdefic}
  and~\ref{th:explicitpres}]\label{lth:graph}
  Suppose that $\Sigma$ is triangle-free. Then for each choice of a maximal
  tree in $\Sigma$ and for each integer~$n$, there is an algorithm that
  produces an explicit presentation for $\Gamma_n$ with $N$ generators and
  $N-1+n(1-\chi(\Sigma))$ relations, where $N$ is the size of the vertex set
  of $\Sigma$. Moreover, these presentations are efficient.
\end{letterthm}

\begin{lettercor}[Corollary~\ref{cor:graphdefic}]\label{lth:graphdefic}
  If $\Sigma$ is triangle-free, then $\defic(\Gamma_n)\to\infty$ as
  $n\to\infty$ if and only if $H_{\Sigma}=\bigcap_n\Gamma_n$ is not finitely
  presentable.
\end{lettercor}

The next theorem shows that there is a logical obstruction to extending
Corollary~\ref{lth:graphdefic} to the case when $\Sigma$ is
higher-dimensional, at least using constructive methods.

\begin{letterthm}[Theorem~\ref{th:noalg}]
  Suppose there is an algorithm that generates a finite presentation
  $\Pres{A(\Sigma)}{R(n,\Sigma)}$ of $\Gamma_n$ for each pair $(n,\Sigma)$
  with $n$ a positive integer and $\Sigma$ a finite flag complex. Suppose
  further that there is a partial algorithm that will correctly determine that
  $\sup_n\abs{R(n,\Sigma)}=\infty$ if $\Sigma$ belongs to a certain collection
  $\mathscr{C}$ of finite flag complexes. Then $\mathscr{C}$ does not coincide
  with the class of $\Sigma$ for which the Bestvina--Brady kernel $H_{\Sigma}$
  is not finitely presentable.
\end{letterthm}

Despite this obstruction, we do discuss the general case in detail: we give a
procedure for building presentations for $\Gamma_n$ and then simplifying them
in the presence of $2$-simplices in $\Sigma$. The results we obtain are
technical to state, but we believe that this part of the paper is in some ways
the most illuminating for understanding why presentations of Artin subgroups
behave as they do.

Rather than cluttering this introduction with a technical result of this
nature, let us instead single out an application of our construction: in the
special case where $\Sigma$ is a topological surface, the presentations we
obtain behave as one expects. By a \emph{standard} flag triangulation, we mean
the disc triangulated as a single $2$-simplex; the sphere triangulated as the
join of a $0$-sphere and a simplicial circle; the projective plane
triangulated as the second barycentric subdivision of a hexagon with opposite
edges identified; or another compact surface with any flag triangulation.

\begin{letterprop}[Proposition~\ref{prop:surfaces}]
  Let $\Sigma$ be a standard flag triangulation of a compact surface and let
  $\Gamma_n\subset G_{\Sigma}$ be the corresponding Artin subgroups. Then the
  algorithm described in \sect\ref{sec:generalsimplify} produces a
  presentation of $\Gamma_n$ with $\abs{\Sigma^{(0)}}$~generators and $R(n)$
  relations, where $\lim_{n\to\infty}R(n)<\infty$ if and only if $\Sigma$ is
  homeomorphic to the disc or the sphere, i.e.\ if and only if $\Sigma$ is
  simply connected.
\end{letterprop}

Finally, in the case when $\Sigma$ is $2$-dimensional, which is particularly
important for the possible application to the relation gap problem discussed
above, we obtain the following general picture:

\begin{letterthm}[Propositions~\ref{prop:goodkernel},~\ref{prop:mv}
  and~\ref{prop:eulchardef}]\label{lth:posres}
  If $\Sigma$ is a finite flag $2$-complex then the following implications
  hold:
  \[\label{eq:bigdiag}
  \xymatrix{
  H_{\Sigma}\text{ is f.\ p.}\ar@{=>}[r]\ar@{<=>}[d] &
    \defic(\Gamma_n)=O(1)\ar@{=>}[d]\\
  \Sigma\text{ is }1\text{-connected}\ar@{=>}[r]\ar@{=>}[d] & \chi(\Sigma)\ge1\\
  \Sigma\text{ is }1\text{-acyclic}\ar@{=>}[ur]\ar@{<=>}[d]\\
  H_{\Sigma}\text{ is }\FPnn{2}\ar@{=>}[r] & \adef(\Gamma_n)=O(1)\ar@{=>}[ul]
  }
  \]
\end{letterthm}

Here is an outline of the paper. In \sect\ref{sec:backgroundBB} we review the
work of Bestvina and Brady on right-angled Artin groups and establish some
notation. In \sect\ref{sec:backgrounddef} we assemble the definitions of
the properties of group presentations that we will consider. 

In \sect\ref{sec:computedefic} we describe an algorithm to construct efficient
presentations of the groups $\Gamma_n$ associated to a triangle-free flag
complex $\Sigma$, and calculate their deficiencies. In
\sect\ref{sec:explicitpres} we implement this algorithm, and write down
explicit presentations for $\Gamma_n$.

The next three sections are concerned with extending the methods and results
of the triangle-free case to a general flag complex. In
\sect\ref{sec:limitations}, we state our main conjecture about the
deficiencies of presentations for general~$\Sigma$, and prove that there is a
recursion-theoretic obstruction to finding an algorithm to construct
presentations verifying this conjecture. In \sect\ref{sec:simplify}, we
present local arguments that let one use the topology of~$\Sigma$ to simplify
presentations of the groups~$\Gamma_n$. This leads to a procedure that
produces presentations of the $\Gamma_n$ by first building large presentations
and then using these local arguments to simplify the presentations by removing
size~$n$ families of relations. We present evidence that the resulting
presentations will be close to realizing the deficiencies of the $\Gamma_n$
and look briefly at presentations for the Bestvina--Brady kernel $H$.
Finally, in \sect\ref{sec:apps} we give one theoretical and one practical
application of the procedure developed in \sect\ref{sec:simplify}: we show how
simplifying the topology of $\Sigma$ by coning off a loop in its $1$-skeleton
leads to a simplification of our presentations of the $\Gamma_n$, and also
work through our procedure in the case when $\Sigma$ is a flag triangulation
of the real projective plane.

We conclude in \sect\ref{sec:diagproof} by using non-constructive methods to
prove results about the deficiencies and abelianized deficiencies of the
$\Gamma_n$ for a general $\Sigma$.

\subsection*{Acknowledgments} Much of the work in this paper formed part of
the second author's Ph.~D. thesis at Imperial College, London, and the
exposition has benefited from the careful reading and useful suggestions of
the two examiners, Ian Leary and Bill Harvey.

\section{The Bestvina--Brady theorem}\label{sec:backgroundBB}
\subsection{Right-angled Artin groups}
Let $\Sigma$ be a finite simplicial complex with vertices $a_1,\ldots,a_N$. We
shall assume that $\Sigma$ is a \emph{flag complex}, i.e.\ that every set of
pairwise adjacent vertices of $\Sigma$ spans a simplex. (We can always arrange
this without altering the topology of $\Sigma$ by barycentrically
subdividing.) We then associate to $\Sigma$ a \emph{right-angled Artin group}
\[
G_{\Sigma}=\left\langle e_1,\ldots,e_N\st [e_i,e_j]\text{ for
}\{a_i,a_j\}\in\Sigma\right\rangle.
\]
\begin{example}
  When $\Sigma$ is a discrete set of $N$ points, $G_{\Sigma}$ is a free group
  of rank~$N$. At the other extreme, when $\Sigma$ is an $(N-1)$-simplex (so
  that the $1$-skeleton of $\Sigma$ is the complete graph on $N$ vertices),
  the corresponding right-angled Artin group $G_{\Sigma}$ is free abelian of
  rank~$N$. If $\Sigma$ is a graph, the flag condition reduces to saying that
  $\Sigma$ is triangle-free, i.e.\ has no cycles of length~$3$.
\end{example}
\begin{remark}
  Of course, $\Sigma$ is determined by its $1$-skeleton, but it is convenient
  to carry along the higher-dimensional cells that make $\Sigma$ into a flag
  complex.
\end{remark}

\subsection{An Eilenberg--Mac~Lane space for $G_{\Sigma}$}\label{sec:EilMac}
Let $e_1,\ldots,e_N$ be an orthonormal basis of Euclidean $N$-space $\E^N$. If
$\sigma=\{a_{i_1},\ldots,a_{i_n}\}$ is a simplex in $\Sigma$, let
$\Box_{\sigma}$ be the regular $n$-cube with vertices at the origin and at
$\sum_{j\in J}e_{i_j}$ for all non-empty $J\subset\{1,\ldots,n\}$.  We define
$K=K_{\Sigma}$ to be the image in $T^N=\E^N/\Z^N$ of
\[
\bigcup\{\Box_{\sigma}\st\sigma\text{ is a simplex in }\Sigma\}.
\]
It is clear that $\pi_1(K_{\Sigma})=G_{\Sigma}$.

\begin{proposition}[Charney--Davis~\cite{CharneyDavis}]\label{prop:KG_RAAG}
  $K$ inherits a locally $\CAT(0)$ metric from the Euclidean metrics on the
  $\Box_{\sigma}$. In particular, $K$ is an Eilenberg--Mac~Lane space for
  $G_{\Sigma}$.
\end{proposition}
\begin{corollary}\label{cor:eulchar}
  $\chi(G_{\Sigma})+\chi(\Sigma)=1$.
  \begin{proof}
    Apart from its single vertex, the cells in $K$ correspond exactly to the
    simplices in $\Sigma$, with a shift in dimension by $1$.
  \end{proof}
\end{corollary}

\subsection{Bestvina--Brady kernels}\label{sec:BBthm}
For each $\Sigma$, there is a surjective homomorphism $\pi:G_{\Sigma}\to\Z$
sending each generator $e_i$ to a fixed generator of $\Z$.
 There are deep links between the topology
of $\Sigma$ and the finiteness properties of the kernel of this map, as the
following theorem reveals.

\begin{theorem}[Bestvina--Brady~\cite{BestvinaBrady}]\label{th:BB}
  Let $\Sigma$ be a finite flag complex and $H$ the kernel of the map
  $G_{\Sigma}\to\Z$.
  \begin{enumerate}
  \item\label{BBfp} $H$ is finitely presented if and only if $\Sigma$ is
    simply connected.
  \item\label{BBFPn} $H$ is of type \FPnn{n+1} if and only if $\Sigma$ is
    $n$-acyclic.
  \item\label{BBFP} $H$ if of type \FP\ if and only if $\Sigma$ is acyclic.
  \end{enumerate}
\end{theorem}

\section{Presentation invariants of groups}\label{sec:backgrounddef}
\noindent In this section, we assemble some basic definitions for later use.
The reader will find a more detailed account of these properties
in~\cite{BTtorsion}.

We shall write $d(\Gamma)$ for the minimum number of elements needed to
generate a group $\Gamma$. If $Q$ is a group acting on $\Gamma$ then we write
$d_Q(\Gamma)$ for the minimum number of $Q$-orbits needed to generate
$\Gamma$.

\subsection{Deficiency and abelianized deficiency}
Let $\Gamma$ be a finitely presented group. The \emph{deficiency} of a finite
presentation $F/R$ of $\Gamma$ is $d_F(R)-d(F)$, where $F$ operates on its
normal subgroup $R$ by conjugation. (Some authors' definition of deficiency
differs from ours by a sign.)

The action of $F$ on $R$ induces by passage to the quotient an action of
$\Gamma$ on the abelianization $\ab{R}$ of $R$, which makes $\ab{R}$ into a
$\ZG$-module, called the \emph{relation module} of the presentation. The
\emph{abelianized deficiency} of the presentation is
$d_{\Gamma}(\ab{R})-d(F)$.

\begin{lemma}[{\cite[Lemma~2]{BTtorsion}}]\label{lem:defdef}
  The deficiency of any finite presentation of $\Gamma$ is bounded below by
  the abelianized deficiency, and this in turn is bounded below by
  $d(H_2(\Gamma))-\rk(H_1(\Gamma))$, where $\rk$ is torsion-free rank.
\end{lemma}

\begin{definition}
  The \emph{deficiency} $\defic(\Gamma)$ (resp.\ \emph{abelianized deficiency}
  $\adef(\Gamma)$) of $\Gamma$ is the infimum of the deficiencies (resp.\
  abelianized deficiencies) of the finite presentations of $\Gamma$.
\end{definition}

\subsection{Efficiency}
Obviously, if $\Gamma$ has a presentation of deficiency
$d(H_2(\Gamma))-\rk(H_1(\Gamma))$ then by Lemma~\ref{lem:defdef} this
presentation realizes the deficiency of the group. In this case, $\Gamma$ is
said to be \emph{efficient}. One knows that inefficient groups exist: R.~Swan
constructed finite examples in~\cite{Swan}, and much later
M.~Lustig~\cite{LustigEff} produced the first torsion-free examples.

\section{Computing the deficiency when $\Sigma$ is triangle-free}%
\label{sec:computedefic}
\noindent Fix a finite connected flag complex $\Sigma$ and let $G$ be the
associated right-angled Artin group, $H$ the kernel of the exponent-sum map
$G\to\Z$, and $\Gamma_n$ the kernel of the corresponding map $G\to\Z/n$. In
this section we shall give a method for building a cellulation of the cover
of the standard Eilenberg--Mac~Lane space for $G$ (see \sect\ref{sec:EilMac})
corresponding to the subgroup $\Gamma_n$. The cellulation is constructed by
repeatedly forming mapping tori, so this exhibits the group $\Gamma_n$ as an
iterated HNN~extension. We use this construction to prove
Proposition~\ref{lth:vertpres}; when $\Sigma$ is $1$-dimensional, counting the
cells in the resulting complex gives Theorem~\ref{lth:graph}.

\subsection{A motivating example} 
Before embarking on the construction we have just described, we believe it
will help the reader if we discuss a simple example that explains our
motivation for proceeding as we do.

Consider, then, the case when $\Sigma$ is a $2$-simplex, so that
\begin{align*}
  G_{\Sigma}&=\Pres{a_1,a_2,a_3}{[a_1,a_2],[a_2,a_3],[a_1,a_3]}\\
  &\cong\Z^3.  
\end{align*}
The map $\pi:G\to\Z$ is given by $a_i\mapsto1$, and $H=\ker\pi$ is isomorphic
to $\Z^2$, which has deficiency $-1$. On the other hand, the finite-index
subgroups $\Gamma_n=\pi^{-1}(n\Z)$ are all isomorphic to $\Z^3$ and have
deficiency $0$.

How can one get presentations realizing the deficiency of $\Gamma_n$? The
standard Eilenberg--Mac~Lane space $K$ in this case is just the $3$-torus
$T^3$ with its usual cubical structure. If we form the $n$-sheeted cover
$\hat{K}$ of $K$ corresponding to the subgroup $\Gamma_n$ of $G$ and contract
a maximal tree, we can use van~Kampen's theorem to read off a presentation of
$\Gamma_n\cong\Z^3$ with $2n+1$ generators and $3n$ relations---far from
realizing the zero deficiency of $\Gamma_n$.

However, if we are prepared to use a different cellulation of $\hat{K}$ then
we can obtain an efficient presentation in the following way. Consider first
the preimage $T(a_1,a_2)\subset \hat{K}$ of the $2$-torus in $K$ spanned by
$a_1$ and $a_2$. This is a single $2$-torus, cellulated as shown in
Figure~\ref{fig:torus}.
\begin{figure}[tb]
  \centering
  \includegraphics{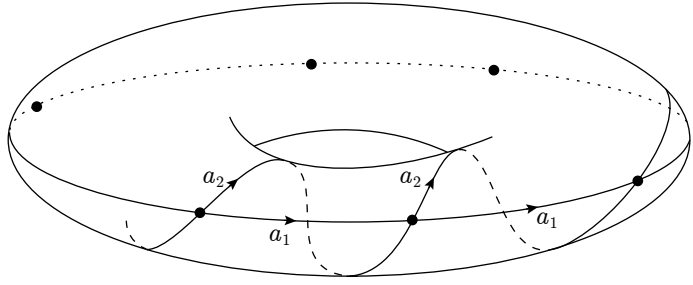}
  \caption{$T(a_1,a_2)$: the preimage in $\hat{K}$ of the torus in $K$
  spanned by $a_1$ and $a_2$.}
  \label{fig:torus}
\end{figure}
Note that this is exactly the mapping torus $\Map(B,\sigma)$ of $B=S^1$
(cellulated with $n$ vertices and $n$ edges labelled $a_1$) by the map
$\sigma$ which is rotation by $2\pi/n$.

Let $\lambda$ be the loop in $B$ spelling out $a_1^n$; this is a generator for
$\pi_1(B)$. Van~Kampen's theorem then gives us a presentation for the
fundamental group of the mapping torus $\Map(B,\sigma)$, namely
$\Pres{\lambda,t}{[\lambda,t]}$, where the stable letter $t$ represents the
path $a_1a_2^{-1}$ in $T(a_1,a_2)$.

Now observe that $\hat{K}$ can be obtained (up to isometry---we are only
changing the cell structure) from $\Map(B,\sigma)$ by again forming a mapping
torus, this time by the shift map $\sigma'$ that rotates each coordinate
circle in $T(a_1,a_2)$ by $2\pi/n$. Again we can compute $\pi_1(\hat{K})$
using van~Kampen's theorem, and we find that
\[
\pi_1(\hat{K})=\Pres{\lambda,t,t'}{[\lambda,t],[\lambda,t'],[t,t']},
\]
where the new stable letter $t'$ is the loop $a_1a_3^{-1}$. This presentation
is  efficient.

\begin{remark}
  To get an efficient presentation, we needed to work with a different
  $\CAT(0)$ cellulation to the obvious lifted cubical structure. In fact, the
  $2$-skeleton of a cube complex $L$ of dimension~$\ge3$ will almost never
  give an efficient presentation for its fundamental group: unless the
  boundary identifications turn it into a torus, any $3$-cube provides an
  obvious essential $2$-sphere in $L^{(2)}$ showing that one of the relations
  is redundant.
\end{remark}

\subsection{The main construction}\label{sec:mainconstruct}
We now turn to the general case when $\Sigma$ is an arbitrary connected finite
flag complex. Let $N$ be the number of vertices in $\Sigma$, and fix an
integer $n\ge1$. We are going to construct an Eilenberg--Mac~Lane space
$Y^n_N$ for the group $\Gamma_n\subset G_{\Sigma}$ as an $N$-fold iterated
mapping torus.

Order the vertices of $\Sigma$, $a_1<a_2<\cdots<a_N$, in such a way
that $a_j$ is contained in the union of the closed stars of the $a_i$ with
$i<j$. Let $Y^n_1$ be a circle, with cell structure consisting of $n$
vertices labelled $0,\ldots,n-1$, and $n$ edges, each labelled $a_1$. We
write $\alpha_1$ for the generator $a_1^n$ of $\pi_1(Y^n_1)$, and take
$\mathscr{P}^n_1=\Pres{\alpha_1}{~}$ as a presentation for $\pi_1(Y^n_1)$.

Suppose inductively that $Y^n_{k-1}$ has been defined and that for $1\le i\le
k-1$, we have a circle $\alpha_i$ in $Y^n_{k-1}$ cellulated as the vertices
$0,\ldots,n-1$ joined in cyclic order by edges labelled $a_i$.  Let
\begin{equation*}
  S_k=\{v\st v<v_k\text{ and }\{v,v_k\}\text{ is an edge in }\Sigma\}
\end{equation*}
and let $B\subset Y_{k-1}^n$ be the union of the loops $\alpha_i$
corresponding to vertices $v_i$ in $S_k$ (Figure~\ref{fig:necklace}), together
with any higher-dimensional cells whose intersection with the $1$-skeleton of
$Y_{k-1}^n$ is contained in this union.
\begin{figure}[tb]
  \centering
  \includegraphics{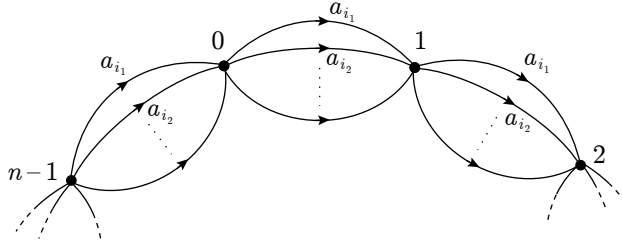}
  \caption{The $1$-skeleton of the base space $B\subset Y^n_{k-1}$ for the
  next mapping torus in our inductive construction.}
  \label{fig:necklace}
\end{figure}

\begin{lemma}\label{lem:baserank}
  If $\Sigma$ is $1$-dimensional, then at the $k$th stage of this process the
  fundamental group of the base space $B$ is free of rank $1+n(\abs{S_k}-1)$.
\end{lemma}
\begin{proof}
  If $\Sigma$ is $1$-dimensional then $B$ is a graph with $n$ vertices and
  $n\abs{S_k}$ edges.
\end{proof}

There is an obvious label-preserving shift map $\sigma:B\to B$, which permutes
the vertices of each $\alpha_i$ as the $n$-cycle $(0~1~\cdots~n-1)$. We glue
the mapping torus of $\sigma$ into $Y^n_{k-1}$ along $B$ to form $Y^n_k$.  In
the obvious cell structure on the mapping torus, there are new $1$-cells
joining $p$ and $p+1$ (mod~$n$) for $p=0,\ldots,n-1$: we label each of these
by $a_k$. Then $\alpha_k=a_k^n$ is a loop in $Y_k^n$. By the
Seifert--van~Kampen theorem, we can obtain a presentation $\mathscr{P}^n_k$
for $\pi_1(Y_k^n)$ by adding to $\mathscr{P}^n_{k-1}$ a single generator
$t_k=a_ka_j^{-1}$ (where $a_j<a_k$ is a choice of vertex containing $a_k$ in
its star), and for each element in a generating set for $\pi_1(B)$ a relation
describing the action of $t_k$ on that generator.

\begin{remark}
  When $\Sigma$ is a graph, this construction of a presentation is really
  algorithmic: indeed, we shall write down the presentations it produces
  explicitly in \sect\ref{sec:explicitpres}. The complication for
  higher-dimensional $\Sigma$ is that the base spaces of the mapping tori will
  be not be as simple as the `necklaces' that occur in the $1$-dimensional
  case: in fact, their fundamental groups can be quite complicated, and
  finding a suitable generating set for these groups is already non-trivial.
  Although presentations of the form we have described still exist abstractly
  in this case, writing them down concretely is no longer straightforward
  (cf.~\sect\ref{sec:limitations}).
\end{remark}

We are now in a position to prove Proposition~\ref{lth:vertpres}.

\begin{proposition}\label{th:vertpres}
  Let $\Sigma$ be a connected finite flag complex. For each integer $n\ge1$,
  there is a generating set for $\Gamma_n=H_{\Sigma}\rtimes n\Z$ indexed by
  the vertices of $\Sigma$, and $\Gamma_n$ cannot be generated by fewer
  elements.
\end{proposition}
\begin{proof}
  The first assertion is immediate from the construction above.
  (Alternatively, one can observe that if $\Sigma^{(0)}=\{a_1,\ldots,a_N\}$
  then $H$ is generated by the elements $a_1a_i^{-1}$ with $i>1$
  and $\Gamma_n$ is generated by these elements together with $a_1^n$.)
  
  To see that $\Gamma_n$ cannot be generated by fewer
  than~$N=\abs{\Sigma^{(0)}}$ elements, one simply observes that
  $\dim_{\Q}H_1(G;\Q)=N$ and $\Gamma_n$ has finite index in $G$, so
  $\dim_{\Q}H_1(\Gamma_n;\Q)\le N$.
\end{proof}

\subsection{Calculating the deficiency}
For the remainder of this section we shall assume that $\Sigma$ is a
\emph{finite flag graph}; we retain the notation of the previous subsection.

\begin{theorem}\label{th:graphdefic}
  The space $Y^n_N$ is non-positively curved and has fundamental group
  $\Gamma_n$. Moreover, $\mathscr{P}^n_N$ is a presentation of $\Gamma_n$ with
  $N$ generators and $N-1+n(1-\chi(\Sigma))$ relations, and this presentation
  is efficient.
\end{theorem}
\begin{proof}
  If we metrize each $\alpha_i$ as a circle of length $n$ then all our gluing
  maps are isometries, and it follows immediately from~\cite[II.11.13]{BH}
  that $Y^n_N$ is non-positively curved. The labels on the $1$-cells of
  $Y^n_N$ show how to define a covering projection from $Y^n_N$ to the
  standard $\KXone{G_{\Sigma}}$; by construction, this cover is
  regular and $\Gamma_n$ is its fundamental group.  Our presentation is
  obtained from a free group by repeated HNN extensions along free subgroups,
  so the corresponding one-vertex $2$-complex is
  aspherical~\cite[Proposition~3.6]{ScottWallart}; indeed, it is visibly
  homotopy equivalent to $Y^n_N$, which is aspherical since it is
  non-positively curved. Consequently the presentation is efficient.
  Finally, by Lemma~\ref{lem:baserank} the number of relations is
  \[
  \sum_{i=2}^N\big(1+n(\abs{S_i}-1)\big)=N-1+n(1-\chi(\Sigma)).
  \]
\end{proof}

\begin{remark}
  The complex $Y^n_{N}$ is homeomorphic to the $n$-sheeted cover of the
  standard cubical $\KXone{G_{\Sigma}}$ with fundamental group $\Gamma_n$, but
  we have given it a different cell structure.
\end{remark}

\begin{corollary}\label{cor:graphdefic}
  When $\Sigma$ is a graph,
  \begin{equation*}
    \defic(\Gamma_n)=1+n(\chi(\Sigma)-1).
  \end{equation*}
  In particular, $\defic(\Gamma_n)\to\infty$ as $n\to\infty$ if and only if
  $H_{\Sigma}$ fails to be finitely presented.
\end{corollary}
\begin{proof}
  Since $\mathscr{P}^n_N$ is efficient, it realizes the deficiency of the
  group.
\end{proof}

\begin{remark}
  The proof of Theorem~\ref{th:graphdefic} breaks down if $\Sigma$ has
  dimension~$>1$: the question of when we obtain presentations realizing the
  deficiencies of the $\Gamma_n$ in this way is a delicate one.
\end{remark}

\section{Explicit presentations when $\Sigma$ is triangle-free}%
\label{sec:explicitpres}
\noindent In \sect\ref{sec:computedefic} we gave an algorithm that, in
principle, one can apply to obtain efficient presentations for the groups
$\Gamma_n$ associated to a finite $1$-dimensional flag complex $\Sigma$. The
purpose of this section is to actually carry out this procedure and thus write
down completely explicit presentations for the groups $\Gamma_n$.

\subsection{A special case}\label{sec:specialcase}
Before presenting the general case, we first describe the particular case when
$\Sigma$ is a cyclic graph of length~$6$, with vertices ordered as shown in
Figure~\ref{fig:triangle}.
\begin{figure}[tb]
  \centering
  \includegraphics{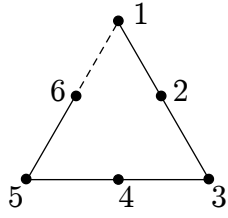}
  \caption{The complex $\Sigma$ in \sect\ref{sec:specialcase}}
  \label{fig:triangle}
\end{figure}
There are two reasons for this: firstly, the notation in the general case is
unwieldy, and it is helpful to first consider this simpler situation, which
contains all the essential ideas; and secondly we shall later build on this
example when we discuss $2$-dimensional flag complexes.

Let us fix $n$ and work through the construction of
\sect\ref{sec:mainconstruct} for the group $\Gamma_n$.

\begin{asparaenum}[\bfseries {Vertex} 1]
\item Our initial space $Y_1$ is the circle with its $n$-vertex cellulation.
  Each edge is labelled $a_1$, and the presentation we take for $\pi_1(Y_1)$
  is
  \[
  \mathscr{P}_1=\Pres{\lambda}{},
  \]
  where $\lambda$ is the loop $a_1^n$. From now on we shall abuse notation and
  write $\lambda=a_1^n$.
\item We form a mapping torus by the degree~$1$ shift map, which rotates the
  base circle by $2\pi/n$; on the fundamental group, this has the effect of an
  HNN~extension with stable letter $t_2=a_2a_1^{-1}$, and our second
  presentation is
  \[
  \mathscr{P}_2=\Pres{\lambda,t_2}{[t_2,\lambda]}.
  \]
\item The base space is the circle $a_2^n=t_2^n\lambda$; our new stable letter
  is $t_3=a_3a_2^{-1}$, and our next presentation is
  \[
  \mathscr{P}_3=\Pres{\lambda,t_2,t_3}{[t_2,\lambda],[t_3,t_2^n\lambda]}.
  \]
\item The base of our mapping torus is now the circle
  \begin{align*}
    a_3^n&=(a_3^na_2^{-n})(a_2^na_1^{-n})a_1^n\\
    &=t_3^nt_2^n\lambda,
  \end{align*}
  and with respect to the stable letter $t_4=a_4a_3^{-1}$ our presentation is
  \[
  \mathscr{P}_4=\Pres{\lambda,t_2,t_3,t_4}{[t_2,\lambda],[t_3,t_2^n\lambda],[t_4,t_3^nt_2^n\lambda]}.
  \]
  Notice that our relations include a path back to the initial vertex of
  $\Sigma$, which acts as a global basepoint.
\item In just the same way, we add another stable letter $t_5=a_5a_4^{-1}$ and
  get a presentation
  \[
  \mathscr{P}_5=\Pres{\lambda,t_2,t_3,t_4,t_5}{[t_2,\lambda],[t_3,t_2^n\lambda],[t_4,t_3^nt_2^n\lambda],[t_5,t_4^nt_3^nt_2^n\lambda]}.
  \]
\item We now come to the final vertex, $a_6$. This has two predecessor
  vertices in our ordering, namely $a_1$ and $a_5$, so the base space
  $B$ for the final mapping torus is a necklace with two strands
  (cf.~Figure~\ref{fig:necklace}): one has edges labelled $a_1$ and the other
  has edges labelled $a_5$. Let us choose and name a set of generators for
  $\pi_1(B)$, which is a free group of rank $n+1$:
  \begin{equation*}
    \begin{aligned}
      w_0&=a_5a_1^{-1}=(a_5a_4^{-1})(a_4a_3^{-1})(a_3a_2^{-1})(a_2a_1^{-1})=t_5t_4t_3t_2\\
      w_1&=a_5w_0a_5^{-1}=(a_5^2a_4^{-2})(a_4^2a_3^{-2})(a_3^2a_2^{-2})(a_2^2a_1^{-2})(a_1a_5^{-1})=t_5^2t_4^2t_3^2t_2^2w_0^{-1}\\
      w_2&=a_5^2w_0a_5^{-2}=t_5^3t_4^3t_3^3t_2^3w_0^{-1}w_1^{-1}\\
      &\ldots\\
      w_{n-1}&=a_5^{n-1}w_0a_5^{1-n}=t_5^nt_4^nt_3^nt_2^nw_0^{-1}w_1^{-1}\cdots
      w_{n-2}^{-1}\\
      \alpha_5&=a_5^n=t_5^nt_4^nt_3^nt_2^n\lambda.
    \end{aligned}
  \end{equation*}
  If we take our stable letter to be $t_6=a_6a_5^{-1}$ then we can work out
  how $t_6$ acts on these generators:
  \begin{equation*}
    \begin{aligned}
    t_6w_0t_6^{-1}&=a_5^{-1}w_0a_5
    =a_5^{-n}a_5^{n-1}w_0a_5^{1-n}a_5^n
    =\alpha_5^{-1}w_{n-1}\alpha_5\\
    t_6w_1t_6^{-1}&=a_5^{-1}a_5w_0a_5^{-1}a_5
    =w_0\\
    &\ldots\\
    t_6w_{n-1}t_6^{-1}&=w_{n-2}\\
    t_6\alpha_5t_6^{-1}&=\alpha_5.
  \end{aligned}
  \end{equation*}
  Therefore our final presentation for $\Gamma_n$ is
  \begin{multline*}
    \mathscr{P}_6=\Pres{\lambda,t_2,t_3,t_4,t_5,t_6}{[t_2,\lambda],[t_3,t_2^n\lambda],[t_4,t_3^nt_2^n\lambda],[t_5,t_4^nt_3^nt_2^n\lambda],\\
    [t_6,t_5^nt_4^nt_3^nt_2^n\lambda],\mathscr{F}_n},
  \end{multline*}
  where
  \[
  \mathscr{F}_n=
  \begin{cases}
    t_6w_0t_6^{-1}=\alpha_5^{-1}w_{n-1}\alpha_5\\
    t_6w_it_6^{-1}=w_{i-1} & (i=1,\ldots,n-1).
  \end{cases}
  \]
\end{asparaenum}

\subsection{The general case}\label{sec:generalcase}
Now let $\Sigma$ be any finite $1$-dimensional flag complex. Choose a maximal
tree $T\subset\Sigma$, and as in \sect\ref{sec:mainconstruct} fix
an ordering $v_1<\cdots<v_N$ of the vertices of $\Sigma$ such that each vertex
(apart from $v_1$) is adjacent to some vertex that precedes it.  Given a pair
of vertices $v$ and $v'$ in $\Sigma$, there is a unique path
$(v=v_{p_1},v_{p_2},\ldots,v_{p_s}=v')$ from $v$ to $v'$ in $T$. Set
\[
\theta^i(v,v')=t_{p_1}^it_{p_2}^i\cdots t_{p_{k-1}}^i.
\]
(For the moment, these are just formal words in the alphabet $\{t_i\}$.)

As always, we start with the presentation $\mathscr{P}_1=\Pres{\lambda}{}$.
Suppose inductively that we have constructed the presentation
$\mathscr{P}_{k-1}$. There is a distinguished predecessor of $v_k$ in our
order, namely the unique vertex $v_{\alpha}<v_k$ such that $v_{\alpha}$ and
$v_k$ are connected by an edge in the maximal tree $T$. There might also be
other vertices $v_{\beta_1},\ldots,v_{\beta_r}$ such that $v_{\beta_i}<v_k$
and $v_{\beta_i}$ is adjacent to $v_k$ in $\Sigma$ (so here $r\ge0$).  The base
space $B$ for the next mapping torus in our construction is now a necklace
with $r+1$ strands, and choosing a presentation for $\pi_1(B)$ amounts to
choosing a maximal tree in $B$. We take this tree to be $a_{\alpha}^{n-1}$,
where $v_{\alpha}$ is our distinguished vertex. For the purposes of our
presentation, the stable letter for the next HNN~extension will be
$t_k=a_ka_{\alpha}^{-1}$.

This gives us the following generating set for $\pi_1(B)$:
\begin{equation*}
  \left.
  \begin{aligned}
    w_{i,0}&=\theta^1(v_k,v_{\beta_i})\\
    w_{i,1}&=\theta^2(v_k,v_{\beta_i})w_{i,0}^{-1}\\
    w_{i,2}&=\theta^3(v_k,v_{\beta_i})w_{i,0}^{-1}w_{i,1}^{-1}\\
    &\;\;\ldots\\
    w_{i,n-1}&=\theta^n(v_k,v_{\beta_i})w_{i,0}^{-1}w_{i,1}^{-1}\cdots w_{i,n-2}^{-1}\\
  \end{aligned}
  \quad\right\}
  \qquad i=1,\ldots r,
\end{equation*}
together with $\gamma_k=a_{\alpha}^n=\theta^n(v_{\alpha},v_1)\lambda$; the
stable letter acts by
\begin{equation*}
  \begin{aligned}
    t_kw_{i,0}t_k^{-1}&=\gamma_k^{-1}w_{i,n-1}\gamma_k\\
    t_kw_{i,j}t_k^{-1}&=w_{i,j-1}\quad(j=1,\ldots,n-1)\\
    [t_k,\gamma_k]&=1.
  \end{aligned}
\end{equation*}
Let $\mathscr{F}_{n,k}$ be the family of relations describing this action of
the stable letter on all the words $w_{i,j}$ for $1\le i\le r$ and $0\le j\le
n-1$, a total of $rn$ relations. (If $v_k$ is not adjacent to any of its
predecessors by an edge not in the tree $T$, then $\mathscr{F}_{n,k}$ is
empty.)  We set
\[
\mathscr{P}_k=\mathscr{P}_{k-1}\cup\Pres{t_k}{[t_k,\gamma_k],\mathscr{F}_{n,k}}.
\]
\begin{theorem}\label{th:explicitpres}
  The presentation $\mathscr{P}_N$ is an efficient presentation of $\Gamma_n$.
\end{theorem}
\begin{proof}
  This is immediate from the construction and Theorem~\ref{th:graphdefic}.
\end{proof}
\begin{remark}
  The choice of a maximal tree in $\Sigma$ is the key technical ingredient
  needed to pass from the special case of~\sect\ref{sec:specialcase} to the
  general case: it provides a consistent way to choose bases for the
  fundamental groups of the successive base spaces for the mapping tori.
\end{remark}

\section{Extension to $2$-complexes: theoretical limitations}%
\label{sec:limitations}
\subsection{The initial aim}
In~\sect\ref{sec:computedefic} we described an algorithm to write down
explicit presentations of the groups $\Gamma_n\subset G_{\Sigma}$ associated
to a $1$-dimensional flag complex $\Sigma$ on a generating set of size
$N=\abs{\Sigma^{(0)}}$ and with $R(n,\Sigma)$ relations, where
$\lim_{n\to\infty}R(n,\Sigma)<\infty$ if and only if the Bestvina--Brady
kernel $H_{\Sigma}$ is finitely presented. One naturally seeks to generalize
this to arbitrary finite flag complexes $\Sigma$ and again construct efficient
presentations of $\Gamma_n$ explicitly; in particular we should like to be
able to prove the following conjecture:
\begin{conjecture}\label{conj:scdef}
  If $\Sigma$ is not simply connected, then $\defic(\Gamma_n)\to\infty$ as
  $n\to\infty$.
\end{conjecture}
The purpose of this section is to show that there is a logical obstruction to
establishing this conjecture simply by constructing explicit presentations
realizing the deficiencies of the $\Gamma_n$.

\subsection{A computability obstruction}
\begin{theorem}\label{th:noalg}
  Suppose there is an algorithm that generates a finite presentation
  $\Pres{A(\Sigma)}{R(n,\Sigma)}$ of $\Gamma_n$ for each pair $(n,\Sigma)$
  with $n$ a positive integer and $\Sigma$ a finite flag complex. Suppose
  further that there is a partial algorithm that will correctly determine that
  $\sup_n\abs{R(n,\Sigma)=\infty}$ if $\Sigma$ belongs to a certain collection
  $\mathscr{C}$ of finite flag complexes. Then $\mathscr{C}$ does not coincide
  with the complement of the class of simply connected finite flag complexes.
\end{theorem}
\begin{proof}
  Suppose that $\mathscr{C}$ is indeed the class of non-simply-connected
  finite flag complexes. Then there is a partial algorithm that takes a finite
  presentation and recognizes that that group it presents is non-trivial: form
  the presentation $2$-complex, barycentrically subdivide until one has a flag
  complex, then apply the partial algorithm hypothesized in the statement of
  the theorem.  But it is well known that no such partial algorithm exists
  (see, for example,~\cite[Corollary~12.33]{Rotman}).
\end{proof}

\subsection{A revised aim}
Although Theorem~\ref{th:noalg} rules out the possibility of a constructive
proof of Conjecture~\ref{conj:scdef} in general, we can nonetheless seek a
widely-applicable and effective procedure that will produce small
presentations of $\Gamma_n$ for a large class $\mathscr{C}$ of $2$-complexes.
In particular, if $\mathscr{C}$ is a class in which triviality of $\pi_1$ can
be algorithmically determined, then we might still hope for a complete
algorithm: for example, if $\mathscr{C}$ is the class of flag triangulations
of compact surfaces, or the class of negatively-curved $2$-complexes. We
discuss the first of these classes in \sect\ref{sec:surfaces}.

\section{Extension to $2$-complexes: simplifying presentations}%
\label{sec:simplify}
\noindent In this section we give examples where one can see explicitly how
simplifying the topology of $\Sigma$ (for example, by adding a $2$-simplex to
kill its boundary loop in $\pi_1$) translates into a simplification of the
presentations of the corresponding groups $\Gamma_n\subset G_{\Sigma}$.

\subsection{Mapping torus constructions for $2$-complexes}%
\label{sec:howhigherdim}
Rather than applying the construction of \sect\ref{sec:mainconstruct} directly
to an arbitrary flag complex, we try to simplify the combinatorics in such a
way that we can build directly on what we have done in the $1$-dimensional
case. Throughout this section, we assume that $\Sigma$ has been obtained by
barycentrically subdividing another complex $\Delta$. (Of course, this has no
significance from the point of view of topology.) Let $\Sigma_1$ be the
subcomplex of the $1$-skeleton of $\Sigma$ obtained by deleting the vertices
at the barycentres of the $2$-simplices of $\Delta$, together with the edges
emanating from these vertices.

We begin by choosing a maximal tree $T$ in $\Sigma_1$ and building a
presentation for $\Gamma_n\subset G_{\Sigma_1}$, exactly as in
\sect\ref{sec:mainconstruct}. This will have a collection of size~$n$ families
of relations, indexed by the edges in the complement of $T$ in $\Sigma_1$.

We now consider adding in the missing vertices, `building over' the
$2$-simplices of $\Delta$. Each $2$-simplex $\sigma$ contributes a $3$-torus
to the standard Eilenberg--Mac~Lane space for our groups, which is exactly the
mapping torus of the degree~$1$ shift map on the union of $2$-tori
corresponding to the edges in the boundary of $\sigma$. In terms of $\pi_1$,
we add a new stable letter $\tau_{\sigma}$ and six relations describing its
action on a generating set for the fundamental group of the base of the
mapping torus. In the edge-path groupoid, $\tau_{\sigma}$ is equal to
$b_{\sigma}a_i^{-1}$ for some $i$, where $b_{\sigma}^n$ is the path traversing
once the copy of $S^1$ corresponding to the vertex at the centre of $\sigma$,
and $a_i$ is one of the vertices in the boundary of $\sigma$.  Another way of
thinking about this is that we are extending our maximal tree from $\Sigma_1$
to a maximal tree for the whole $1$-skeleton of $\Sigma$, adjoining the edges
$(b_{\sigma},a_i)$.

We are going to examine how these extra relations can be used to eliminate
size~$n$ families of relations picked up in the first part of the
construction, which only involved $\Sigma_1$.

\subsection{A local argument: completing a 2-simplex}%
\label{sec:completesimplex}
We start by discussing in detail the simplest case, namely when $\Sigma$ is a
$2$-simplex, barycentrically subdivided. Then $\Sigma_1$ is exactly the flag
graph we analysed in \sect\ref{sec:specialcase}, so we will adopt the same
notation as there: order the vertices in the boundary of the simplex as in
Figure~\ref{fig:triangle}, with the $i$th vertex called $a_i$, and let $Y_6$ be
the $2$-complex we finished up with at the end of \sect\ref{sec:specialcase}.
We shall call the extra vertex at the barycentre of the $2$-simplex $b$; it
will come after all the $a_i$ in our order.

We can make a $3$-complex isometric to the cover of the standard
Eilenberg--Mac~Lane space corresponding to the index~$n$ subgroup
$\Gamma_n\subset G_{\Sigma}$ by forming the mapping torus of $Y_6$ by the map
$\sigma$ that acts as a degree~$1$ shift on each circle $a_i^n$ in $Y_6$.

The base $Y_6$ is a union of six tori, glued along coordinate circles. Its
fundamental group is generated by $\lambda,t_2,\ldots,t_6$, and we can get a
presentation for the fundamental group $\Gamma_n$ of the mapping torus of
$\sigma$ by saying how a stable letter $\tau=ba_5^{-1}$ acts on these
generators. For example, one calculates that
\begin{align*}
  \tau t_4\tau^{-1}&=ba_5^{-1}a_4a_3^{-1}a_5b^{-1}\\
  &=a_4a_5^{-1}a_3^{-1}a_5\\
  &=t_5^{-1}(a_3^{-1}a_4)(a_4^{-1}a_5)\\
  &=t_5^{-1}t_4t_5,
\end{align*}
and similarly for the other generators. The resulting presentation is
\begin{multline*}
  \Gamma_n=\Pres{\lambda,t_2,t_3,t_4,t_5,t_6,\tau}{[t_2,\lambda],[t_3,t_2^n\lambda],[t_4,t_3^nt_2^n\lambda],[t_5,t_4^nt_3^nt_2^n\lambda],\\
  [t_6,t_5^nt_4^nt_3^nt_2^n\lambda],\mathscr{F}_n,\mathscr{G}},
\end{multline*}
where $\mathscr{F}_n$ is as in~\sect\ref{sec:specialcase} and
\[
\mathscr{G}=
\begin{cases}
  [\tau,t_6],\\
  [\tau,t_5],\\
  \tau t_4\tau^{-1}=t_5^{-1}t_4t_5,\\
  \tau t_3\tau^{-1}=t_5^{-1}t_4^{-1}t_3t_4t_5,\\
  \tau t_2\tau^{-1}=t_5^{-1}t_4^{-1}t_3^{-1}t_2t_3t_4t_5,\\
  \tau\lambda\tau^{-1}=t_6\lambda t_6^{-1}.
\end{cases}
\]
We know from Theorem~\ref{lth:posres} that the groups $\Gamma_n$ have
presentations in which the number of relations does not go to infinity with
$n$, so for large~$n$ this presentation will have many superfluous relations.
Our aim is to see how one can eliminate the size $n$ family $\mathscr{F}_n$.
Specifically, we will prove:
\begin{proposition}
  $\Gamma_n$ admits the following presentation:
  \begin{multline*}
    \Gamma_n=\Pres{\lambda,t_2,t_3,t_4,t_5,t_6,\tau}{[t_2,\lambda],[t_3,t_2^n\lambda],[t_4,t_3^nt_2^n\lambda],[t_5,t_4^nt_3^nt_2^n\lambda],\\
    [t_6,t_5^nt_4^nt_3^nt_2^n\lambda],\mathscr{G},[t_6^{-1}\tau,w_0]}.
  \end{multline*}
\end{proposition}
\begin{corollary}
  $\defic(\Gamma_n)\le5$.
\end{corollary}

The proposition is a consequence of the following three lemmas.
\begin{lemma}\label{lem1}
  The relation
  \begin{equation*}
    [t_6^{-1}\tau,w_0]=1\tag{$*$}
  \end{equation*}
  holds in $\Gamma_n$.
\end{lemma}
\begin{proof}
  One has
  \begin{align*}
    t_6^{-1}\tau
    w_0\tau^{-1}t_6&=
    t_6^{-1}(t_5^{\tau})(t_4^{\tau})(t_3^{\tau})(t_2^{\tau})t_6\\
    &=
    t_6^{-1}t_5t_5^{-1}t_4t_5t_5^{-1}t_4^{-1}t_3t_4t_5t_5^{-1}t_4^{-1}t_3^{-1}t_2t_3t_4t_5t_6\\
    &= t_6^{-1}t_2t_3t_4t_5t_6\\
    &= a_5a_1^{-1}\\
    &= w_0.
  \end{align*}
\end{proof}

\begin{remark}
  We are using here the geometric interpretation of $\Gamma_n$ as
  $\pi_1(Y_6)$, and working in the edge-path groupoid of $Y_6$. We need to do
  this (and pick up the relation in the lemma) because the raw presentation
  does not tell us how $t_6$ acts on the lower $t_i$, although we know how
  $\tau$ acts on them.  This is, in a sense, the key point: $t_6$ acts on
  $T(a_1,a_5)$, whereas $\tau$ acts on the whole base $B$. Before, we had to
  spell out how $t_6$ acts on all the $w_i=w_i(\lambda,t_2,\ldots,t_5)$, but
  now we can make do with a single relation saying that its action is the same
  as the action of $\tau$: we know how $\tau$ acts on the $w_i$ because we
  know how it acts on all the $t_i$.
\end{remark}

\begin{lemma}
  Let $\mathscr{F}_n'$ be the family obtained from $\mathscr{F}_n$ by replacing
  all occurrences of $t_6$ with $\tau$. The relations $\mathscr{F}_n'$ follow
  from $\mathscr{G}$.
\end{lemma}
\begin{proof}
  This is a routine inductive calculation. For example,
  \begin{align*}
    \tau w_1\tau^{-1}&= \tau
    t_5^2t_4^2t_3^2t_2t_3^{-1}t_4^{-1}t_5^{-1}\tau^{-1}\\
    &=t_5^2t_5^{-1}t_4^2t_5t_5^{-1}t_4^{-1}t_3^2t_4t_5t_5^{-1}t_4^{-1}t_3^{-1}t_2^2t_3t_4t_5t_5^{-1}t_4^{-1}t_3^{-1}t_2^{-1}t_3t_4t_5\\ 
    &\hphantom{==}t_5^{-1}t_4^{-1}t_3^{-1}t_4t_5t_5^{-1}t_4^{-1}t_5t_5^{-1}\quad\text{
    (making }\tau\text{ act throughout)}\\
    &=t_5t_4t_3t_2\\
    &=w_0.
  \end{align*}
\end{proof}

\begin{lemma}\label{lem3}
  The relations $\mathscr{F}_n$ follow from $\mathscr{F}_n'$, the relation
  $(*)$, and $[t_6,\alpha_5]$.
\end{lemma}
\begin{proof}
  Again, this is no more than a simple combinatorial calculation. That $t_6$
  acts as desired on $w_0$ is immediate from $(*)$, and one can then prove
  the same for $w_{n-k}$ by induction on $k$: thus,
  \begin{align*}
    t_6w_{n-1}t_6^{-1}&= t_6\alpha_5\tau w_0\tau^{-1}\alpha_5^{-1}t_6^{-1} &
    \text{using }\mathscr{F}_n'\\
    &= \tau\alpha_5t_6w_0t_6^{-1}\alpha_5^{-1}\tau^{-1} & \text{since
    }t_6\text{ commutes with }\alpha_5\text{ and }\tau\\
    &=\tau w_{n-1}\tau^{-1} & \text{by the inductive hypothesis}\\
    &=w_{n-2} & \text{using }\mathscr{F}_n'\text{ again}
  \end{align*}
  and similarly for $w_{n-2},\ldots,w_1$, except that $\alpha_5$ no longer
  appears.
\end{proof}

\subsection{A local argument: completing a $2$-simplex with two `missing
edges'}\label{sec:completetwoedge}
We now show how two of our size~$n$ families of relations can be consolidated
into a single family in the presence of a $2$-simplex.

Let us consider, then, the complex $\Sigma$ shown in
Figure~\ref{fig:twotriangles}.
\begin{figure}[tb]
  \centering
  \includegraphics{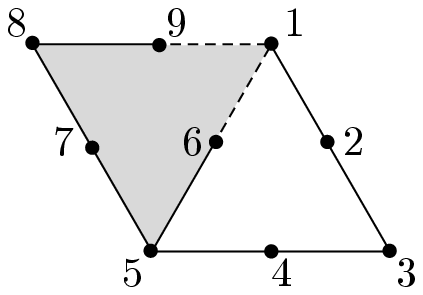}
  \caption{The complex $\Sigma$ in~\sect\ref{sec:completetwoedge}. We will
  fill in the left triangle with a (subdivided) $2$-simplex.}
  \label{fig:twotriangles}
\end{figure}
We order the vertices as indicated in the figure. Following through our
procedure once again, we obtain a presentation $\mathscr{P}_9$ for $Y_9$. We
are going to modify this presentation slightly by performing a Tietze move: we
shall introduce a new generator $\theta$, which we set equal to the word
$t_6t_5t_4t_3t_2$ in the other generators. (In the edge-path groupoid,
$\theta$ is equal to $a_6a_1^{-1}$.) The resulting presentation is
\begin{align*}
  \mathscr{P}=\Pres{\lambda,t_2&,t_3,t_4,t_5,t_6,t_7,t_8,t_9,\theta}{\theta=t_6t_5t_4t_3t_2,\\
  &[t_2,\lambda],[t_3,t_2^n\lambda],[t_4,t_3^nt_2^n\lambda],[t_5,t_4^nt_3^nt_2^n\lambda],[t_6,t_5^nt_4^nt_3^nt_2^n\lambda],\mathscr{F}_n,\\
  &[t_7,t_5^nt_4^nt_3^nt_2^n],[t_8,t_7^nt_5^nt_4^nt_3^nt_2^n\lambda],[t_9,t_8^nt_7^nt_5^nt_4^nt_3^nt_2^n\lambda],\mathscr{H}_n},
\end{align*}
where if
\begin{equation*}
  \begin{aligned}
    w_0&=t_5t_4t_3t_2\\
    w_1&=t_5^2t_4^2t_3^2t_2^2w_0^{-1}\\
    w_2&=t_5^3t_4^3t_3^3t_2^3w_0^{-1}w_1^{-1}\\
    &\;\;\vdots\\
    w_{n-1}&=t_5^nt_4^nt_3^nt_2^nw_0^{-1}w_1^{-1}\cdots
    w_{n-2}^{-1}\\
    \alpha&=t_5^nt_4^nt_3^nt_2^n\lambda
  \end{aligned}
\end{equation*}
and
\begin{equation*}
  \begin{aligned}
    u_0&=t_8t_7t_5t_4t_3t_2\\
    u_1&=t_8^2t_7^2t_5^2t_4^2t_3^2t_2^2u_0^{-1}\\
    &\;\;\vdots\\
    u_{n-1}&=t_8^nt_7^nt_5^nt_4^nt_3^nt_2^nu_0^{-1}u_1^{-1}\cdots
    u_{n-2}^{-1}\\
    \beta&=t_8^nt_7^nt_5^nt_4^nt_3^nt_2^n\lambda
  \end{aligned}
\end{equation*}
then
\[
\mathscr{F}_n=
\begin{cases}
  t_6w_0t_6^{-1}=\alpha^{-1}w_{n-1}\alpha\\
  t_6w_it_6^{-1}=w_{i-1} & (i=1,\ldots,n-1).
\end{cases}
\]
and
\[
\mathscr{H}_n=
\begin{cases}
  t_9u_0t_9^{-1}=\beta^{-1}u_{n-1}\beta\\
  t_9u_it_9^{-1}=u_{i-1} & (i=1,\ldots,n-1).
\end{cases}
\]
(cf.~\sect\ref{sec:completesimplex}).

We now complete $\Sigma$ by filling in the left-hand triangle with a
(subdivided) $2$-simplex. As before, we obtain a $3$-dimensional
$\KXone{\Gamma_n}$ by forming the mapping torus of $Y_9$ by the degree-$1$
shift map on the $a_i^n$ for $i=1,5,6,7,8,9$. A set of $\pi_1$~generators for
the base of this mapping torus is given by
\begin{equation*}
  \begin{aligned}
    a_9a_8^{-1}&=t_9\\
    a_8a_7^{-1}&=t_8\\
    a_7a_5^{-1}&=t_7\\
    a_6a_5^{-1}&=t_6\\
    a_6a_1^{-1}&=\theta\\
    a_1^n&=\lambda,
  \end{aligned}
\end{equation*}
and by the Seifert--van~Kampen theorem we obtain a presentation for $\Gamma_n$
by adding to our presentation a stable letter $\tau$, and a collection of
relations describing the action of $\tau$ on these generators. One checks
easily that these relations are
\[
\mathscr{G}=
\begin{cases}
  [\tau,t_9],\\
  [\tau,t_8],\\
  \tau t_7\tau^{-1}=t_8^{-1}t_7t_8,\\
  \tau t_6\tau^{-1}=t_8^{-1}t_7^{-1}t_6t_7t_8,\\
  \tau\theta\tau^{-1}=t_8^{-1}t_7^{-1}t_6\theta t_6^{-1}t_7t_8,\\
  \tau\lambda\tau^{-1}=t_8^{-1}t_7^{-1}t_6\lambda t_6^{-1}t_7t_8.
\end{cases}
\]

\begin{lemma}
  The relation
  \begin{equation*}
    [t_9^{-1}\tau,u_0]=1\tag{$*$}
  \end{equation*}
  holds in $\Gamma_n$.
\end{lemma}
\begin{proof} The proof is entirely analogous to that
  of Lemma~\ref{lem1}.
\end{proof}

\begin{lemma}
  Let $\mathscr{H}'_n$ be the family of words obtained from $\mathscr{H}_n$ be
  replacing all occurrences of $t_9$ with $\tau$. The relations
  $\mathscr{H}'_n$ follow from $\mathscr{F}_n$ and $\mathscr{G}$.
\end{lemma}
\begin{proof}
  This is another straightforward induction: we present a sample
  calculation. Observe that $u_1=t_8^2t_7^2w_1t_7^{-1}t_8^{-1}$, and
  $w_1=t_6^{-1}w_0t_6=t_6^{-2}\theta t_6$. Thus
  \[
  \begin{aligned}
    \tau
    u_1\tau^{-1}&=(t_8^{\tau})^2(t_7^{\tau})^2(t_6^{\tau})^{-2}\theta^{\tau}t_6^{\tau}(t_7^{\tau})^{-1}(t_8^{\tau})^{-1}\\
    &=t_8^2t_8^{-1}t_7^2t_8t_8^{-1}t_7^{-1}t_6^{-2}t_7t_8t_8^{-1}t_7^{-1}t_6\theta
    t_6^{-1}t_7t_8t_8^{-1}t_7^{-1}t_6t_7t_8t_8^{-1}t_7^{-1}t_8^{-1}\\
    &=t_8t_7t_6^{-1}\theta\\
    &=u_0.
  \end{aligned}
  \]
\end{proof}

\begin{proposition}
  The presentation obtained from $\mathscr{P}$ by replacing
  $\mathscr{H}_n$ by the relation $(*)$ is again a presentation of $\Gamma_n$.
\end{proposition}
\begin{proof}
  Given the previous two lemmas, it suffices to show that the relations
  $\mathscr{H}_n$ follow from $\mathscr{H}'_n$, the relation $(*)$, and the
  other relations in $\mathscr{P}$. Just as in Lemma~\ref{lem3}, this can be
  proved by a straightforward induction.
\end{proof}

\subsection{A scheme for simplifying presentations for a general
$2$-complex}\label{sec:generalsimplify}
Now consider a general $\Sigma$, obtained as before by barycentrically
subdividing another complex $\Delta$.  Recall that we have a presentation of
$\Gamma_n\subset G_{\Sigma}$ with a size~$n$ family of relations for each edge
in the complement of a maximal tree $T$ in $\Sigma_1\subset\Sigma$, and for
each $2$-simplex $\sigma$ in $\Delta$ we have a stable letter $\tau_{\sigma}$
and six relations describing its action on a generating set for the
fundamental group of the base of the corresponding mapping torus.

The local arguments of \sect\ref{sec:completesimplex} and
\sect\ref{sec:completetwoedge} yield the following procedure.

\begin{proposition}\label{prop:procedure}
  One can eliminate (by which we mean replace by a single relation of the form
  $[\tau_{\sigma}t_i,w_0]=1$) the family of relations corresponding to an edge
  $e$ in $\Sigma_1-T$ if $\Delta$ contains a $2$-simplex the boundary of whose
  image in $\Sigma$ is contained in $T\cup e$, or in $T\cup e\cup f$ (in this
  case, we keep the family of relations corresponding to $f\subset
  \Sigma_1-T$).
\end{proposition}

By repeated application of this proposition, we can potentially remove a
family of relations from our presentation for each $2$-simplex of $\Delta$;
thus by the end, we have $\ge1-\chi(\Delta)$ families. In particular, if
$\chi(\Delta)<1$ then we obtain a presentation for $\Gamma_n$ with $O(n)$
relations, as we expect. The interesting case is when $\chi(\Delta)=1$: in
this case it is possible that we may be able to remove all of the size~$n$
families. One might hope that this happens if and only if $H_{\Delta}$ is
finitely presented, but we know from Theorem~\ref{th:noalg} that we cannot
hope to create an algorithm exhibiting this, even if is true.

\subsection{Presentations of the Bestvina--Brady kernels}
An entirely analogous procedure to the one described above lets one obtain
presentations for the Bestvina--Brady kernels themselves. In fact, we can
derive these presentations purely formally from the ones we have found for
$\Gamma_n$ by setting $n=\infty$ and discarding the generator $\lambda$.

It is interesting to compare the resulting presentations with those found by
W.~Dicks and I.~Leary:

\begin{theorem}[\cite{DicksLeary}]
  If $\Sigma$ is connected, then the group $H_{\Sigma}$ has a presentation
  with generating set the directed edges of $\Sigma$, and relators all words
  of the form $e_1^ne_2^n\cdots e_k^n$ for $(e_1,\ldots,e_k)$ a directed cycle
  in $\Sigma$.
\end{theorem}

Our presentations are instead on a generating set indexed by the vertices of
$\Sigma$, though in fact these generators are better thought of as directed
edges: the generator associated to a vertex $v$ really corresponds to the edge
in our chosen maximal tree from $v$ to its immediate predecessor.  As
$n\to\infty$, our families become infinite families, whose elements contain
the characteristic subwords $t_{i_1}^n\cdots t_{i_k}^n$ for all integers $n$,
where $(t_{i_1},\ldots,t_{i_k})$ is a path in the $1$-skeleton of $\Sigma$.

Both presentations of the kernel bring their own insights: in the
Dicks--Leary presentation, one can see very clearly how non-trivial loops in
$\pi_1(\Sigma)$ lead to infinite families of relations in the kernel, whereas
our presentation has the benefit of being on a minimal generating set.

\section{Extension to $2$-complexes: applications}\label{sec:apps}
\subsection{Coning off an edge loop}\label{sec:coneoff}
For our first application, we show how killing a loop in $\pi_1(\Sigma)$ can
allow us to remove a family of relations.

\begin{proposition}\label{prop:coneoff}
  Let $\hat{\Sigma}$ be a complex obtained from $\Sigma$ by coning off an
  edge-loop $\ell$ in $\Sigma_1$, and let
  $\hat{\Gamma}_n=H_{\hat{\Sigma}}\rtimes n\Z$.  Let $e\subset\Sigma_1$ be an
  edge in $\ell$ that does not lie in the maximal tree $T$. From the
  presentations of $\Gamma_n$ described in \sect\ref{sec:howhigherdim}, one
  can obtain presentations of $\hat{\Gamma}_n$ such that for each~$n$, the
  family of relations corresponding to $e$ is removed and a fixed number of
  relations independent of $n$ is added.
\end{proposition}
\begin{proof}
  Extend the maximal tree in $\Sigma_1$ to a maximal tree $\hat{T}$ in
  $\hat{\Sigma}_1$ in such a way that it contains both `halves' of the
  subdivided edge $\alpha$ from the cone point to $\partial e$
  (Figure~\ref{fig:coneoff}).
  \begin{figure}[tb]
    \centering
    \includegraphics{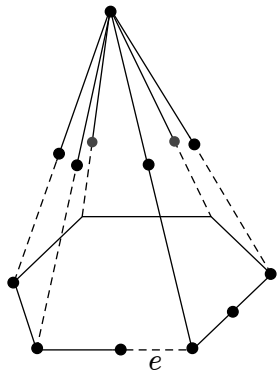}
    \caption{Coning off a loop in $\Sigma_1$.}
    \label{fig:coneoff}
  \end{figure}
  We now start filling in $2$-simplices around the cone, starting with the
  one containing $\alpha$ but not $e$ in its boundary (for example, in
  Figure~\ref{fig:coneoff}, this is the front-right simplex).  In each case,
  we are able, using Proposition~\ref{prop:procedure}, to remove the family
  attached to the `missing' edge between the cone point and $\Sigma_1$. When
  we come around to the last $2$-simplex, we have already eliminated this
  `vertical' family, so we can instead use the $2$-simplex to eliminate the
  family corresponding to the edge $e$, again by
  Proposition~\ref{prop:procedure}.
\end{proof}
\begin{remark}
  One can view our argument that adding a subdivided $2$-simplex lets us
  remove one of the families on the boundary of that $2$-simplex as a special
  case of Proposition~\ref{prop:coneoff}: the subdivided $2$-simplex is
  exactly the cone on its boundary.
\end{remark}

\begin{remark}
  This argument shows that the number of size~$n$ families of relations that
  we are left with in our presentation is bounded below by the \emph{killing
  number} of $\pi_1(\Sigma)$: recall that a group $\Gamma$ has killing number
  $\le k$ if $\Gamma$ is the normal closure of $k$ elements.
  In~\cite{Wiegold}, J.~Lennox and J.~Wiegold prove that any finite perfect
  group has killing number~$1$, and conjecture that the same is true for any
  finitely generated perfect group.
\end{remark}

\subsection{When $\Sigma$ is a triangulated surface}\label{sec:surfaces}
This simple situation illustrates clearly how a homotopically non-trivial loop
can prevent one from removing all the families of relations.

\begin{proposition}\label{prop:surfaces}
  Let $\Sigma$ be a standard flag triangulation of a compact surface and let
  $\Gamma_n$ be the corresponding Artin subgroups. Then repeated use of
  Proposition~\ref{prop:procedure} produces a presentation of $\Gamma_n$ with
  $O(1)$ relations if and only if $\Sigma$ is homeomorphic to the disc or the
  sphere, i.e.~if and only if $\Sigma$ is simply connected.
\end{proposition}

\begin{remark}
  It is interesting to compare this to the subdivided $2$-simplex we analysed
  in~\sect\ref{sec:completesimplex}: in each case, the Euler characteristic of
  the complex is~$1$, but whereas in the former example our procedure allowed
  us to remove the single size~$n$ family of relations in our presentation, in
  this case the topology prevents us from doing this. This underscores the
  fact that the presentation invariants we are examining are more subtle than
  just Euler characteristic, which in turn encourages the belief that the
  asymptotic behaviour of $\defic(\Gamma_n)$ should not depend only on the
  homology of $\Sigma$.
\end{remark}

\begin{proof}[Proof of Proposition~\ref{prop:surfaces}]
For surfaces of Euler characteristic~$\le0$, our procedure gives us no hope of
removing all of the infinite families. On the other hand, we have already seen
in \sect\ref{sec:completesimplex} that the groups $\Gamma_n$ associated to the
triangulation of the disc as a single $2$-simplex can be presented with $O(1)$
relations, while in any standard triangulation of the sphere the equator is
coned off (cf.~Proposition~\ref{prop:coneoff}). Thus it only remains to
consider the projective plane.

The `standard' flag triangulation $\Sigma$ of $\R P^2$ is obtained from the
complex $\Sigma_1$ shown in Figure~\ref{fig:rp2} by filling in (subdivisions
of) all the visible triangles.
\begin{figure}[tb]
  \centering
  \includegraphics{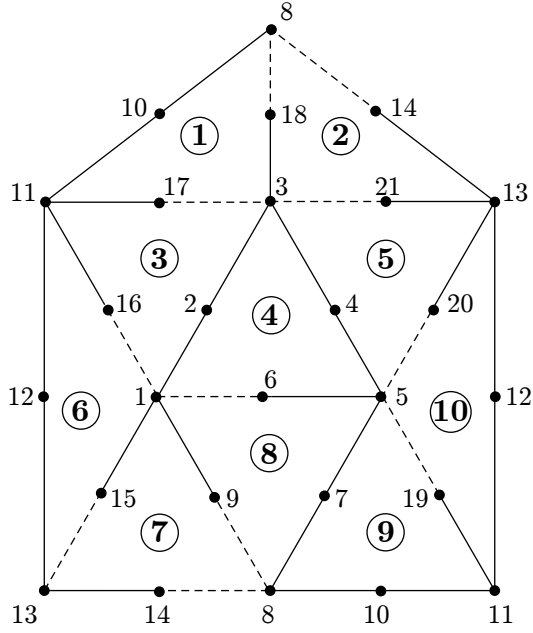}
  \caption{$\Sigma_1$ when $\Delta$ is a triangulation of $\R P^2$.}
  \label{fig:rp2}
\end{figure}
We order the vertices as shown; we choose the maximal tree whose complement is
the dotted lines. One checks that $\Sigma_1$ has $21$ vertices and $30$ edges,
so $\rk(\pi_1(\Sigma_1))=10$, and thus in our presentation for $\Gamma_n$ we
have $10$ size~$n$ families of relations. Let us write $\mathscr{F}(a,b)$ for
the family corresponding to the edge joining the vertices labelled $a$ and
$b$.

To form $\Sigma$, we must now adjoin $10$ $2$-simplices; we shall refer to
them by boldface numbers, as in Figure~\ref{fig:rp2}. Now we start to apply
the simplification procedure described in Proposition~\ref{prop:procedure}:
\begin{equation*}
  \begin{aligned}
    \mathbf{4}&\Rightarrow \mathscr{F}(1,6)\\
    \mathbf{9}&\Rightarrow\mathscr{F}(5,19)\\
    \mathbf{8}+\mathscr{F}(1,6)&\Rightarrow\mathscr{F}(8,9)\\
    \mathbf{10}+\mathscr{F}(5,19)&\Rightarrow\mathscr{F}(5,20)\\
    \mathbf{5}+\mathscr{F}(5,20)&\Rightarrow\mathscr{F}(3,21)
  \end{aligned}
\end{equation*}
At this point we can make no further progress. However, if we agree to keep
$\mathscr{F}(8,14)$ in our presentation (note that this corresponds to the
edge completing a loop generating $\pi_1(\Sigma)$, namely the loop around
the boundary in Figure~\ref{fig:rp2}) then we can use this to dispose of
the remaining families:
\begin{equation*}
  \begin{aligned}
    \mathbf{2}+\mathscr{F}(3,21)+\mathscr{F}(8,14)&\Rightarrow\mathscr{F}(8,18)\\
    \mathbf{1}+\mathscr{F}(8,18)&\Rightarrow\mathscr{F}(3,17)\\
    \mathbf{3}+\mathscr{F}(3,17)&\Rightarrow\mathscr{F}(1,16)\\
    \mathbf{6}+\mathscr{F}(1,16)&\Rightarrow\mathscr{F}(13,15)
  \end{aligned}
\end{equation*}
\end{proof}

\section{Proof of Theorem~\ref{lth:posres}}\label{sec:diagproof}
\noindent In this section we move away from our emphasis on constructive
methods, and use other techniques to establish Theorem~\ref{lth:posres}.  The
non-trivial left-hand vertical implications in the diagram on
page~\pageref{eq:bigdiag} follow from the Bestvina--Brady Theorem (see
Theorem~\ref{th:BB}); in this section, we prove the remaining non-trivial
implications.

\subsection{When the kernel has good finiteness properties}\label{sec:goodker}
The following lemma is well-known (see, for example,~\cite[VIII.5,
Exercise~3b]{Brown}).

\begin{lemma}\label{lem:fp2relmod}
  A finitely generated group $\Gamma$ is of type~\FPnn{2} if and only if every
  relation module for a presentation of $\Gamma$ on a finite generating set is
  finitely generated as a $\ZG$-module.
\end{lemma}

\begin{proposition}\label{prop:goodkernel}
  \begin{enumerate}
  \item If $H$ is finitely presented, then $\defic(\Gamma_n)$ is bounded
    uniformly in $n$.
  \item If $H$ is of type~\FPnn{2}, then $\adef(\Gamma_n)$ is bounded
    uniformly in $n$.
  \end{enumerate}
  \begin{proof}
    We have $\Gamma_n = H\rtimes n\Z$. Thus if one has a presentation
    $H=F/R$, where $F$ is a finitely generated free group with basis elements
    $a_i$ say, then $\Gamma_n$ can be presented by augmenting this
    presentation by a generator $t_n$ for $n\Z$ and one relation of the
    form $t_na_it_n^{-1}=u_{i,n}$ for each $a_i$, where $u_{i,n}\in F$. Let
    $S$ be this additional set of relations and let $R'$ be the normal closure
    of $R\cup S$ in $F'=F*\angles{t_n}$.  Since $\Gamma_n = F'/R'$, the first
    assertion follows.

    If $H$ is of type ${\rm FP}_2$, then $R/[R,R]$ is finitely generated as a
    $\Z F$ module, and adding the image of $S$ to any finite generating set
    for $R/[R,R]$ provides a generating set for $R'/[R',R']$ as a $\Z F'$
    module, which proves the second part.
  \end{proof}
\end{proposition}

\subsection{A Mayer--Vietoris argument}\label{sec:MV}
Again we begin by recalling an easy lemma; one can find a proof
in~\cite[Proposition~5.4]{Brown} or~\cite{BTtorsion}.
\begin{lemma}\label{lem:relmodexact}
  If $F/R$ is a presentation for $\Gamma$ with $F$ free of rank $r$, then
  there is an exact sequence of $\ZG$-modules
 \[
 0\to\ab{R}\to(\ZG)^r\to\ZG\to\Z\to0.
 \]
\end{lemma}
\begin{proposition}\label{prop:mv}
  Let $\Sigma$ be a finite, connected, flag $2$-complex. Suppose that $\Sigma$
  is not $1$-acyclic. Then $\adef(\Gamma_n)\to\infty$ as $n\to\infty$.
  \begin{proof}
    By choosing a free $\ZG$-module mapping onto $\ab{R}$, and splicing this
    surjection to the exact sequence of Lemma~\ref{lem:relmodexact} in
    degrees~$\le1$, one can obtain a partial free resolution of $\Z$ over
    $\ZG$. In the light of Proposition~\ref{lth:vertpres}, one sees that in
    order to show that $\adef(\Gamma_n)\to\infty$ it therefore suffices to
    show that $d(H_2(\Gamma_n))\to\infty$ as $n\to\infty$.

    To prove this, we adopt the Morse theory point of view developed
    in~\cite{BestvinaBrady}. Let $K$ be the standard Eilenberg--Mac~Lane space
    for $G_{\Sigma}$, and recall that $K$ is a subcomplex of $\E^N/\Z^N$,
    where $N=\abs{\Sigma^{(0)}}$. The Bestvina--Brady Morse function
    $f:\lift{K}\to\R$ lifts to universal covers the function $K\to S^1$
    induced by the coordinate-sum map $\E^n\to\R$.

    Let $S$ be the infinite cyclic cover of $K$ corresponding to the
    subgroup $H=\ker\pi$ of $G$. There is a height function $S\to\R$ induced
    by $f$; we again denote this by $f$.

    For $n\ge3$, let $S_n$ be the finite vertical strip $f^{-1}[-n,n]$.
    Observe that
    \begin{align*}
      \varinjlim H_2(S_n)&=H_2(\varinjlim S_n)\\
      &=H_2(S)\\
      &=H_2(H_{\Sigma}).
    \end{align*}
    By~\cite[Corollary~7]{LearySaad}, $H_2(H_{\Sigma})$ maps surjectively onto
    an infinite direct sum of copies of the non-trivial abelian group
    $H_1(\Sigma)$, and it follows that $d(H_2(S_n))\to\infty$ as $n\to\infty$.

    Let $T_n$ be $S_n$ with its two ends glued together by the deck
    transformation $\phi$ with $f\circ\phi(s)=f(s+2n)$.  Let $A_n$ and $B_n$
    be respectively the images in $T_n$ of $f^{-1}[1-n,n-1]$ and $f^{-1}\left(
    [-n,2-n]\cup[n-2,n] \right)$. Let $C_n=A_n\cap B_n$, so that $C_n$ is
    homeomorphic to a disjoint union of two copies of $f^{-1}[0,2]$. We have
    just seen that $d(H_2(A_n))\to\infty$ as $n\to\infty$; moreover, $B_n$
    and $C_n$ have $H_1$ and $H_2$ generated by a number of elements
    independent of~$n$. But there is an exact Mayer--Vietoris sequence
    \[
    \cdots\to H_2(C_n)\to H_2(A_n)\oplus H_2(B_n)\to H_2(T_n)\to
    H_1(C_n)\to\cdots,
    \]
    so it follows that $d(H_2(T_n))\to\infty$, i.e.\
    $d(H_2(\Gamma_n))\to\infty$, as required.
  \end{proof}
\end{proposition}

\subsection{Using Euler characteristic}\label{sec:euler}
\begin{proposition}\label{prop:eulchardef}
  Suppose that $\Sigma$ is $2$-dimensional. If $\chi(\Sigma)<1$ then
  $\defic(\Gamma_n)\to\infty$.
  \begin{proof}
    By Corollary~\ref{cor:eulchar}, we have that $\chi(G)>0$, and since
    $(G:\Gamma_n)=n$ it follows that $\chi(\Gamma_n)\to\infty$. On the other
    hand, there is a finite $3$-dimensional $\KXone{\Gamma_n}$-complex $K$,
    and given any presentation $2$-complex $L$ for $\Gamma_n$, say with $g$
    $1$-cells and $r$ $2$-cells, we can make a complex $L'$ homotopy
    equivalent to $K$ by attaching cells in dimensions $\ge3$.  It follows
    that
    \begin{align*}
      \chi(\Gamma_n)&=\chi(K)\\
      &=1-b_1(K)+b_2(K)-b_3(K)\\
      &\le1-b_1(L)+b_2(L)\\
      &=\chi(L)\\
      &=1-g+r,
    \end{align*}
    so $\defic(\Gamma_n)\to\infty$.   
  \end{proof}
\end{proposition}

\nocite{*}
\bibliography{books}

\providecommand{\bysame}{\leavevmode\hbox to3em{\hrulefill}\thinspace}
\providecommand{\MR}{\relax\ifhmode\unskip\space\fi MR }
\providecommand{\MRhref}[2]{%
  \href{http://www.ams.org/mathscinet-getitem?mr=#1}{#2}
}
\providecommand{\href}[2]{#2}
\begin{thebibliography}{10}

\bibitem{BestvinaBrady}
M.~Bestvina and N.~Brady, \emph{Morse theory and finiteness properties of
  groups}, Invent. Math. \textbf{129} (1997), 445--470.

\bibitem{BKS}
M.~Bestvina, B.~Kleiner, and M.~Sageev, \emph{The asymptotic geometry of
  right-angled {A}rtin groups, {I}}, preprint (2007).

\bibitem{BradyMeier}
N.~Brady and J.~Meier, \emph{Connectivity at infinity for right angled {A}rtin
  groups}, Trans. Amer. Math. Soc. \textbf{353} (2001), 117--132.

\bibitem{BH}
M.~R. Bridson and A.~Haefliger, \emph{Metric spaces of non-positive curvature},
  Grundlehren der Mathematischen Wissenschaften, vol. 319, Springer-Verlag,
  1999.

\bibitem{BTtorsion}
M.~R. Bridson and M.~Tweedale, \emph{Deficiency and abelianized deficiency of
  some virtually free groups}, Math. Proc. Camb. Phil. Soc. (to appear).

\bibitem{Brown}
K.~S. Brown, \emph{Cohomology of groups}, Graduate Texts in Mathematics,
  vol.~87, Springer-Verlag, 1982.

\bibitem{CCV}
R.~Charney, J.~Crisp, and K.~Vogtmann, \emph{Automorphisms of $2$-dimensional
  right-angled {A}rtin groups}, Geom. Topol. (to appear).

\bibitem{CharneyDavis}
R.~Charney and M.~W. Davis, \emph{Finite {$K(\pi, 1)$}s for {A}rtin groups},
  Prospects in topology (Princeton, NJ, 1994), Ann. of Math. Stud., vol. 138,
  Princeton Univ. Press, 1995, pp.~110--124.

\bibitem{DicksLeary}
W.~Dicks and I.~J. Leary, \emph{Presentations for subgroups of {A}rtin groups},
  Proc. Amer. Math. Soc. \textbf{127} (1999), 343--348.

\bibitem{Humphries}
S.~P. Humphries, \emph{On representations of {A}rtin groups and the {T}its
  conjecture}, J. Algebra \textbf{169} (1994), 847--862.

\bibitem{JensenMeier}
C.~Jensen and J.~Meier, \emph{The cohomology of right-angled {A}rtin groups
  with group ring coefficients}, Bull. London Math. Soc. \textbf{37} (2005),
  711--718.

\bibitem{KimRoush}
K.~H. Kim and F.~W. Roush, \emph{Homology of certain algebras defined by
  graphs}, J. Pure Appl. Algebra \textbf{17} (1980), 179--186.

\bibitem{LearySaad}
I.~J. Leary and M.~Saadeto\u{g}lu, \emph{The cohomology of {B}estvina--{B}rady
  groups}, preprint (2006).

\bibitem{Wiegold}
J.~C. Lennox and J.~Wiegold, \emph{Generators and killers for direct and free
  products}, Arch. Math. (Basel) \textbf{34} (1980), 296--300.

\bibitem{LustigEff}
M.~Lustig, \emph{Non-efficient torsion-free groups exist}, Comm. Algebra
  \textbf{23} (1995), 215--218.

\bibitem{PapSuciu}
S.~Papadima and A.~I. Suciu, \emph{Algebraic invariants for right-angled
  {A}rtin groups}, Math. Ann. \textbf{334} (2006), 533--555.

\bibitem{Rotman}
J.~J. Rotman, \emph{An introduction to the theory of groups}, 4th ed., Graduate
  Texts in Mathematics, vol. 148, Springer-Verlag, 1995.

\bibitem{ScottWallart}
P.~Scott and C.~T.~C. Wall, \emph{Topological methods in group theory},
  Homological group theory (Proc. Sympos., Durham, 1977), London Math. Soc.
  Lecture Note Ser., vol.~36, Cambridge Univ. Press, 1979, pp.~137--203.

\bibitem{Swan}
R.~G. Swan, \emph{Minimal resolutions for finite groups}, Topology \textbf{4}
  (1965), 193--208.

\end{thebibliography}
\bibliographystyle{amsplain}
\end{document}